\documentclass[twoside,a4paper,reqno,11pt]{amsart}
\usepackage{amsfonts, amsbsy, amsmath, amssymb, latexsym}
\usepackage{mathrsfs,array}
\usepackage[top=30mm,right=30mm,bottom=30mm,left=30mm]{geometry}
\usepackage{stmaryrd}
\usepackage{bm}
\usepackage{xcolor}
\usepackage{hyperref}
\usepackage{tikz-cd}
\usepackage{relsize}

\makeatletter
\@namedef{subjclassname@2020}{\textup{2020} Mathematics Subject Classification}
\makeatother

\def\thefootnote{\fnsymbol{footnote}}

 \renewcommand{\to}{\rightarrow}

\DeclareMathOperator{\GL}{GL}

\DeclareMathOperator{\Rad}{Rad}

\makeatletter
\newcommand{\imod}[1]{\allowbreak\mkern4mu({\operator@font mod}\,\,#1)}
\makeatother

\newtheorem{theorem}{Theorem}[section]
\newtheorem{lemma}[theorem]{Lemma}
\newtheorem{proposition}[theorem]{Proposition}

\newtheorem{corollary}[theorem]{Corollary}

\theoremstyle{definition}

\newtheorem{example}[theorem]{Example}

\begin{document}

\title{The hypercenter of an algebraic group}
\author[D.\ Sercombe]{Damian Sercombe}
\address
{Mathematisches Institut, Universität Freiburg, Ernst-Zermelo-Straße 1, 79104 Freiburg, Germany}
\email{damian.sercombe@math.uni-freiburg.de}

\begin{abstract} We show that any connected algebraic group $G$ over a field admits a nilpotent normal subgroup $Z_\infty(G)$ such that the quotient $G/Z_\infty(G)$ has trivial center. We construct $Z_\infty(G)$ as the final term of the transfinitely extended upper central series of $G$; accordingly, we call it the hypercenter of $G$. We establish several related results about the upper central series of $G$, along with an analogue for algebraic groups of a well-known theorem of Fitting's \cite{Fi}.
\end{abstract}

\let\thefootnote\relax

\maketitle

\section{Introduction}\label{intro}

\noindent The hypercenter of a (possibly infinite, abstract) group $G$ is defined to be the final term of the transfinitely extended upper central series of $G$. It is well-studied and interesting subgroup of $G$, with several equivalent characterisations (see for instance \cite{Ba} or \cite[Ch.\ 1, 2]{R}). 
One such characterisation is that the hypercenter of $G$ equals the intersection of all normal subgroups $N$ of $G$ such that the center of $G/N$ is trivial.

\vspace{2mm}\noindent For some time the author has been interested in finding an analogue of this notion for an algebraic group $G$, by which we mean a group scheme of finite type over a field. This question arose from working on various generation problems for algebraic groups $G$, which turn out to be much easier for the case where $G$ is centerless, so we wanted to be able to reduce these problems to the centerless case.

\vspace{2mm}\noindent In this paper we find what we were looking for. Our main result is that any connected algebraic group $G$ is an extension of an algebraic group with trivial center, by a nilpotent algebraic group. We call this nilpotent normal subgroup the hypercenter $Z_\infty(G)$ of $G$. We construct $Z_\infty(G)$ by extending the upper central series of $G$ by transfinite induction, observing that this transfinitely extended series terminates, and then taking its final term. 
We also characterise $Z_\infty(G)$ and show that it is invariant under base change.

\vspace{2mm}\noindent We prove our main result via several intermediate results, including an analogue of a result of Fitting's \cite[Thm.\ 14]{Fi}. Our proof relies upon \cite[Thm.\ 1.1]{S}, which was only published in 2026. We conclude with an example which shows that the connectedness assumption on $G$ is essential.

\vspace{2mm}\noindent Our setup henceforth is as follows. Let $k$ be a field. 
Let $G$ be an algebraic $k$-group, i.e. a group scheme of finite type over $k$ (we do not assume that $G$ is smooth, nor affine). We will almost always be interested in the case where $G$ is connected. A subgroup of $G$ will refer to a locally closed $k$-subgroup scheme.

\vspace{2mm}\noindent Our main result is the following:

\begin{theorem}\label{centerlessbynilpotent} Let $G$ be a connected algebraic $k$-group. There exists a nilpotent normal subgroup $Z_\infty(G)$ of $G$ such that the center of $G/Z_\infty(G)$ is trivial. The formation of $Z_\infty(G)$ commutes with base change by algebraic field extensions.
\end{theorem}

\noindent We construct this subgroup $Z_\infty(G)$ as the final term of the transfinitely extended upper central series of $G$. As such, we call it the \textit{hypercenter} of $G$. This subgroup $Z_\infty(G)$ may be characterised as the intersection of all normal subgroups $N$ of $G$ such that the center of $G/N$ is trivial, we prove this in Lemma \ref{hypercenterchar}.

\vspace{2mm}\noindent We use two main ingredients to prove Theorem \ref{centerlessbynilpotent}. The first is the following result, which is an analogue of a well-known theorem of Fitting's \cite[Thm.\ 14]{Fi} for finite (abstract) groups.

\begin{theorem}\label{largestnilpotentnormsubgp} Let $G$ be a connected algebraic $k$-group. There exists a largest nilpotent normal subgroup $F(G)$ of $G$ (i.e. a nilpotent normal subgroup of $G$ which contains all other nilpotent normal subgroups of $G$). 
\end{theorem}

\noindent Continuing the analogy with finite groups, we call this subgroup $F(G)$ in Theorem \ref{largestnilpotentnormsubgp} the \textit{Fitting subgroup} of $G$. Our proof of Theorem \ref{largestnilpotentnormsubgp} uses the recent result \cite[Thm.\ 1.1]{S}.

\vspace{2mm}\noindent The second ingredient in our proof is showing that the upper central series of $G$ is quite well-behaved, as illustrated in Theorems \ref{hypercenterthm} and \ref{UCSterminatesthm}. We now briefly define this series up to the first infinite ordinal $\omega$; for more details refer to \S \ref{upcenser}.

\vspace{2mm}\noindent The (\textit{non-extended}) \textit{upper central series} $\{Z_i(G)\}_{i\geq 0}$ of an algebraic $k$-group $G$ is an ascending chain of characteristic subgroups of $G$ which is uniquely defined by the following condition: $Z_0(G)$ is trivial and, for all integers $i\geq 0$, $Z_{i+1}(G)/Z_i(G)=Z(G/Z_i(G))$. We define the \textit{$\omega$'th center} of $G$ to be the schematic union $$Z_\omega(G):=\mathsmaller{\bigcup}_{i\geq 0}\,Z_i(G),$$ by which we mean the smallest closed subscheme of $G$ containing $Z_i(G)$ for every integer $i\geq 0$. 

\vspace{2mm}\noindent Henceforth assume that $G$ is connected. We now establish some properties of the $\omega$'th center $Z_\omega(G)$ of $G$.

\begin{theorem}\label{hypercenterthm} Let $G$ be a connected algebraic $k$-group.

\vspace{1mm}\noindent (a) $Z_\omega(G)$ is a nilpotent normal subgroup of $G$.

\vspace{1mm}\noindent (b) The formation of $Z_\omega(G)$ commutes with base change by algebraic field extensions.

\vspace{1mm}\noindent (c) Suppose $G$ is affine and $Z(G)$ is unipotent. Then $Z_\omega(G)$ is unipotent.

\vspace{1mm}\noindent (d) $Z_{\omega}(G/Z_\omega(G))$ is unipotent. 
\end{theorem}

\begin{corollary}\label{hypercentercor1} Let $G$ be a connected algebraic $k$-group. Then, for all integers $i\geq 1$, $Z_{\omega}(G/Z_i(G))$ is unipotent.
\end{corollary}

\noindent According to \cite[IV, \S 3, Thm.\ 1.1(a)]{DG} any commutative affine algebraic $k$-group $H$ admits a largest multiplicative type subgroup, which we denote by $H_s$.

\begin{corollary}\label{hypercentercor2} Let $G$ be a connected affine algebraic $k$-group. Then $Z_{\omega}(G/Z(G)_s)$ is unipotent.
\end{corollary}

\noindent We may extend the upper central series of $G$ to higher ordinals via transfinite recursion, as follows. For each ordinal $\alpha$ we define a subgroup $Z_\alpha(G)$ of $G$ by the following rules: \begin{itemize}\item $Z_0(G)$ is trivial, \vspace{1mm}\item $Z_{\alpha}(G)/Z_{\alpha-1}(G)=Z(G/Z_{\alpha-1}(G))$ if $\alpha$ is a successor ordinal, and \vspace{0.8mm}\item $Z_\alpha(G)=\bigcup_{\beta<\alpha}Z_\beta(G)$ if $\alpha$ is a limit ordinal.\end{itemize} The (\textit{transfinitely extended}) \textit{upper central series} of $G$ is the normal series \begin{equation}\label{transuppercentserorig} \{Z_\alpha(G)\}\,,\end{equation} where $\alpha$ runs over all of the ordinals. [N.B. For this definition to make sense we need to check that each $Z_\alpha(G)$ is indeed a normal subgroup of $G$, see \S \ref{upcenser} for more details.]

\begin{theorem}\label{UCSterminatesthm} Let $G$ be a connected algebraic $k$-group.

\vspace{1mm}\noindent (a) The (transfinitely extended) upper central series $(\ref{transuppercentserorig})$ of $G$ terminates at an ordinal which is strictly less than $\omega^2$. 
In particular, the length of $(\ref{transuppercentserorig})$ is countable.

\vspace{1mm}\noindent (b) For any ordinal $\alpha$, the formation of $Z_\alpha(G)$ commutes with base change by algebraic field extensions.
\end{theorem}

\noindent We conclude with the observation that Theorem \ref{centerlessbynilpotent} fails without the assumption of $G$ being connected. We illustrate this in Example \ref{etaleunion}, where we take $G=\mathbb{G}_m\rtimes\mathbb{Z}/2\mathbb{Z}$ with non-trivial associated action.

\section{Preliminaries}\label{preliminaries}

\noindent Let $k$ be a field, say of characteristic $p\geq 0$. 

\vspace{2mm}\noindent By an \textit{algebraic $k$-group} we mean a group scheme of finite type over $k$. Unless explicitly stated, we do not assume that algebraic $k$-groups are smooth, nor affine. By a \textit{subgroup} $H$ of an algebraic $k$-group $G$ we mean a locally closed $k$-subgroup scheme; note that $H$ is automatically closed by \cite[047T]{Stacks} and is of finite type by \cite[01T5, 01T3]{Stacks}. The \textit{image} of a homomorphism of algebraic $k$-groups refers to the schematic image, in the sense of \cite[01R7]{Stacks}.

\vspace{2mm}\noindent For the case where $p>0$, we denote by $\alpha_p$ and $\mu_p$ the first Frobenius kernels of the additive group $\mathbb{G}_a$ and the multiplicative group $\mathbb{G}_m$ respectively.

\subsection{Schematic unions}\label{schunions}\textcolor{white}{a}

\vspace{2mm}\noindent Let us recall the definition of the schematic union.

\vspace{2mm}\noindent Let $X$ be a scheme. Let $\mathcal{Y}$ be a set of closed subschemes of $X$. [N.B. No set-theoretic issues arise here, as the category of schemes is well-powered.] The category of schemes admits all (even infinite) coproducts, so we may form the coproduct scheme $\coprod_{Y\in\mathcal{Y}} Y$. Let $\xi:\coprod_{Y\in\mathcal{Y}} Y \to X$ be the morphism of schemes given by inclusion on each component. We then define the \textit{schematic union} $\bigcup_{Y\in\mathcal{Y}}Y$ of $\mathcal{Y}$ in $X$ to be the (schematic) image of $\xi$.

\vspace{2mm}\noindent If $X$ is affine then $\bigcup_{Y\in\mathcal{Y}}Y$ is the closed subscheme of $X$ whose associated ideal is the intersection of those of every $Y\in\mathcal{Y}$. 

\vspace{2mm}\noindent In the following situation, we show that the formation of the schematic union commutes with base change. Recall that $k$ is a field.

\begin{proposition}\label{schunionbasechange} Let $X$ be an affine scheme of finite type over $k$. Let $\mathcal{Y}$ be a (possibly infinite) set of closed subschemes of $X$. Consider the schematic union $\bigcup_{Y\in\mathcal{Y}}Y$ of $\mathcal{Y}$ in $X$. Let $K/k$ be an algebraic field extension. Then $(\bigcup_{Y\in\mathcal{Y}}Y)_K=\bigcup_{Y\in\mathcal{Y}}Y_K$. 
\begin{proof} The inclusion $\bigcup_{Y\in\mathcal{Y}}Y_K\subseteq(\bigcup_{Y\in\mathcal{Y}}Y)_K$ is clear, as by definition $\bigcup_{Y\in\mathcal{Y}}Y_K$ is the smallest closed subscheme of $X_K$ which contains $Y_K$ for every $Y\in\mathcal{Y}$. Now for the opposite inclusion.

\vspace{2mm}\noindent Let $\mathcal{O}(X)$ be the coordinate ring of $X$; it is a finitely generated $k$-algebra. Consider the extension of scalars functor $(-)\otimes_k K$ from the category of finitely generated $k$-algebras to the category of finitely generated $K$-algebras. Since $K/k$ is a field extension, this functor $(-)\otimes_k K$ is exact. 
For any closed subscheme $Z$ of $X$, let $I(Z)$ denote the ideal of $\mathcal{O}(X)$ corresponding to $Z$. It follows from the exactness of $(-)\otimes_k K$ that \begin{equation}\label{compat} I(Z_K)=I(Z)\otimes_k K.\end{equation} 
By definition of the schematic union, we have \begin{equation}\label{defschunion} I(\mathsmaller{\bigcup}_{Y\in\mathcal{Y}}Y)=\mathsmaller{\bigcap}_{Y\in\mathcal{Y}}I(Y).\end{equation}

\vspace{2mm}\noindent For the moment assume that $K/k$ is finite. It then follows from \cite[\S 7.6, Lem.\ 1, Thm.\ 4]{BLR} that the functor $(-)\otimes_k K$ has a left adjoint; namely, Weil restriction. 
Hence $(-)\otimes_k K$ preserves all categorical limits; in particular, all (possibly infinite) intersections. Combining this with Equations $(\ref{compat})$ and $(\ref{defschunion})$ gives us $$I(\mathsmaller{\bigcup}_{Y\in\mathcal{Y}}Y_K)
=\mathsmaller{\bigcap}_{Y\in\mathcal{Y}}(I(Y)\otimes_k K)=(\mathsmaller{\bigcap}_{Y\in\mathcal{Y}}I(Y))\otimes_k K
=I((\mathsmaller{\bigcup}_{Y\in\mathcal{Y}}Y)_K).$$ Therefore $\bigcup_{Y\in\mathcal{Y}}Y_K=(\bigcup_{Y\in\mathcal{Y}}Y)_K$.

\vspace{2mm}\noindent We now relax the assumption that $K/k$ is finite. Denote $Z:=\bigcup_{Y\in\mathcal{Y}}Y_K$. Let $k'/k$ be the minimal field of definition for $Z$ in $X_K$, in the sense of \cite[Def.\ 1.1.6]{CGP}. Let $Z'$ be the $k'$-descent of $Z$ in $X_{k'}$. We claim that $k'/k$ is finite. For the moment, assume the claim holds. We already proved the finite case, so $(\bigcup_{Y\in\mathcal{Y}}Y)_{k'}=\bigcup_{Y\in\mathcal{Y}}Y_{k'}$. By definition $\bigcup_{Y\in\mathcal{Y}}Y_{k'}$ is the smallest closed subscheme of $X_{k'}$ which contains $Y_{k'}$ for every $Y\in\mathcal{Y}$. But $Z'$ also contains every $Y_{k'}$, hence $\bigcup_{Y\in\mathcal{Y}}Y_{k'}\subseteq Z'$. Then extending scalars by $K/k'$ gives us $$(\mathsmaller{\bigcup}_{Y\in\mathcal{Y}}Y)_K=(\mathsmaller{\bigcup}_{Y\in\mathcal{Y}}Y_{k'})_K\subseteq Z.$$

\vspace{2mm}\noindent It remains to prove the claim. 
Let $J$ be the ideal of $\mathcal{O}(X)\otimes_k K$ associated to $Z$. By assumption $\mathcal{O}(X)$ is a finitely generated $k$-algebra. 
So $\mathcal{O}(X)$, and hence $\mathcal{O}(X)\otimes_k K$, is a Noetherian ring. Let $\{a_1,...,a_n\}$ be a finite generating set for $J$ as an $(\mathcal{O}(X)\otimes_k K)$-module. 
For each $i=1,...,n$ write $a_i=x_i\otimes\alpha_i$, where $x_i\in\mathcal{O}(X)$ and $\alpha_i\in K$. Let $L:=k(\alpha_1,...,\alpha_n)$. Of course $a_i\in\mathcal{O}(X)\otimes_k L$ for every $i$, so $J$ descends to an ideal of $\mathcal{O}(X)\otimes_k L$. Translating this to affine schemes, $Z$ descends to a closed subscheme of $X_L$. Hence $k'\subseteq L$. But $L/k$ is finite, since each $\alpha_i$ is algebraic over $k$. 
Therefore $k'/k$ is finite, proving the claim.
\end{proof}
\end{proposition}

\noindent We now turn our attention to algebraic groups.

\vspace{2mm}\noindent Henceforth let $G$ be an affine algebraic $k$-group, and let $\{H_i \hspace{0.5mm}|\hspace{0.5mm} i\in I\}$ be a -- possibly uncountable -- set of (closed) subgroups of $G$ that is upward directed under inclusion (i.e. $I$ is an upward directed poset where, for $i,j\in I$, $i\leq j$ if and only if $H_i\subseteq H_j$). Let $H:=\bigcup_{i\in I}H_i$, the schematic union of the $H_i$'s in $G$. It was shown in \cite[Prop.\ 2.2]{S} that $H$ is a subgroup of $G$.

\vspace{2mm}\noindent The \textit{nilpotency class} of a nilpotent algebraic $k$-group refers to the (finite) length of its upper central series. For more details see \S\ref{upcenser}.

\begin{proposition}\label{nilclasspresschunion} Let $c\geq 0$ be an integer. Suppose, for every $i\in I$, that $H_i$ is nilpotent with nilpotency class at most $c$. Then $H$ is nilpotent with nilpotency class at most $c$.
\begin{proof} We first consider the case where $I$ is a (possibly uncountable) chain. We proceed by induction on $c$. If $c=0$ then each $H_i$ is trivial, hence $H$ is also trivial. Now assume that $c>0$.

\vspace{2mm}\noindent Taking the centraliser in $G$ of every element of $\{H_i \hspace{0.5mm}|\hspace{0.5mm}i\in I\}$ gives us a descending chain \begin{equation}\label{centchain} \big(Z_G(H_i)\big)_{i\in I}\, .\end{equation} Since $G$ is a Noetherian scheme, 
it satisfies the descending chain condition on closed subschemes. So this chain $(\ref{centchain})$ stabilises. 
That is, there exists a sufficiently large $r\in I$ such that $Z_G(H_i)=Z_G(H_r)$ for all $i\in I$ with $i \geq r$. Consequently, $Z_G(H)=Z_G(H_r)$. 
It follows that \begin{equation}\label{centers} Z(H_i)=H_i\cap Z(H)\end{equation} for every $i\geq r$. 

\vspace{2mm}\noindent Consider the natural projection $\rho:H\to H/Z(H)$. By definition of the schematic union, observe that \begin{equation}\label{imunion}\rho(H)=\mathsmaller{\bigcup}_{i\in I}\rho(H_i).\end{equation} 
Using Equation $(\ref{centers})$, for every $i\geq r$ we have \begin{equation}\label{rhoHi}\rho(H_i)
\cong H_i/(H_i \cap Z(H))=H_i/Z(H_i).\end{equation}

\vspace{1mm}\noindent By assumption, for every $i\in I$, $H_i$ is nilpotent with nilpotency class at most $c$. Therefore, by Equation $(\ref{rhoHi})$, for every $i\geq r$ (and hence every $i\in I$) we see that $\rho(H_i)$ is nilpotent with nilpotency class at most $c-1$. Then combining Equation $(\ref{imunion})$ with the inductive hypothesis tells us that $\rho(H)$ is nilpotent with nilpotency class at most $c-1$. Hence $H$ is nilpotent with nilpotency class at most $c$. 

\vspace{2mm}\noindent It remains to consider the general case; that is, $I$ is an arbitrary upward directed poset. Let $\mathcal{P}$ be the poset of nilpotent subgroups of $G$ with nilpotency class at most $c$, under inclusion. We already showed that every chain in $\mathcal{P}$ admits a supremum; namely, its schematic union in $G$. 
It is a basic set-theoretic fact that, given a poset in which every chain admits a supremum, then every upward directed subset admits a supremum. 
Hence this is true of $\mathcal{P}$. In particular, our upward directed subset $\{H_i \hspace{0.5mm}|\hspace{0.5mm}i\in I\}$ of $\mathcal{P}$ admits a supremum in $\mathcal{P}$; let us denote it by $\overline{H}$.

\vspace{2mm}\noindent By definition of the schematic union, $H$ is contained in $\overline{H}$. But we already know that $H$ is a subgroup of $G$, so $H$ is nilpotent with nilpotency class at most $c$ (and thus $H=\overline{H}$).
\end{proof}
\end{proposition}

\begin{corollary}\label{commpresschunion} Suppose $H_i$ is commutative for every $i\in I$. Then $H$ is commutative.
\begin{proof} This is just Proposition \ref{nilclasspresschunion} for the case $c=1$.
\end{proof}
\end{corollary}

\subsection{The upper central series}\label{upcenser}\textcolor{white}{a}

\vspace{2mm}\noindent The upper central series of an algebraic $k$-group $G$ is defined inductively as follows. 

\vspace{2mm}\noindent Denote $G^0:=G$, and let $\zeta_0:G\to G$ be the identity map. For each integer $i\geq 1$ let $G^i:=G^{i-1}/Z(G^{i-1})$, and define \begin{equation}\label{zetai} \zeta_i:G\to G^i\end{equation} to be the composition of $\zeta_{i-1}:G\to G^{i-1}$ with the natural projection $G^{i-1}\to G^i$.

\vspace{2mm}\noindent For each integer $i\geq 0$, the kernel $\ker\zeta_i=:Z_i(G)$ is called the \textit{$i$'th center} of $G$. Each $Z_i(G)$ is a characteristic subgroup of $G$, see for instance \cite[IV, \S4, 1.3, Rem.\ 1.16]{DG}. Note that $Z_1(G)$ is just the usual center $Z(G)$ of $G$. For ease of notation we will usually refer to the natural projection map $\zeta_1:G\to G/Z(G):=G^1$ by simply $\zeta$.

\vspace{2mm}\noindent The central series \begin{equation}\label{uppercentser} 1=Z_0(G)\subseteq Z_1(G)\subseteq Z_2(G)\subseteq ... \end{equation} is called the (\textit{non-extended}) \textit{upper central series} of $G$. 

\vspace{2mm}\noindent Equivalently, the upper central series $(\ref{uppercentser})$ is the unique ascending central series of $G$ which satisfies the following condition: $Z_0(G)$ is trivial and, for all integers $i\geq 0$, $Z_{i+1}(G)/Z_i(G)=Z(G/Z_i(G))$.

\vspace{2mm}\noindent We say that $G$ is \textit{nilpotent} if this series $(\ref{uppercentser})$ terminates after finitely many terms at $G$. If $G$ is nilpotent then its \textit{nilpotency class} is the smallest integer $c$ such that $G=Z_c(G)$. These definitions, along with some characterisations, may be found in \cite[VIB, Def.\ 8.2.1, Prop.\ 8.2]{SGA3}. 

\vspace{2mm}\noindent Let $\omega$ denote the first infinite ordinal. 
We define the \textit{$\omega$'th center} of $G$ to be the schematic union \begin{equation}\label{omegacenterdef} Z_\omega(G):=\mathsmaller{\bigcup}_{i\geq 0}\,Z_i(G).\end{equation}

\vspace{1mm}\noindent For the case where $G$ is connected, we establish some properties of $Z_\omega(G)$ in Theorem \ref{hypercenterthm}. This will become a key ingredient in our proof of Theorem \ref{centerlessbynilpotent}. 

\vspace{2mm}\noindent Henceforth assume that $G$ is connected. We now extend the upper central series of $G$ to higher ordinals via transfinite recursion. This process is analogous to the well-known construction of the transfinite upper central series of an abstract group. Indeed, one could generalise both constructions by relaxing the ``finite type" and ``connected" assumptions on $G$. 

\vspace{2mm}\noindent Associated to every ordinal $\alpha$ is a normal subgroup $Z_\alpha(G)$ of $G$, which we call the \textit{$\alpha$'th center} of $G$. It is defined inductively by the following conditions: \begin{itemize}\item $Z_0(G)$ is trivial, \vspace{1mm}\item  $Z_{\alpha+1}(G)/Z_{\alpha}(G)=Z(G/Z_{\alpha}(G))$ for each ordinal $\alpha$, and \vspace{0.8mm}\item $Z_\gamma(G)=\bigcup_{\beta<\gamma}Z_\beta(G)$ for each limit ordinal $\gamma$ (where $\bigcup$ refers to the schematic union in $G$).\end{itemize}

\vspace{1mm}\noindent In order for this definition to make sense, we need to check that each $Z_\alpha(G)$ is indeed a normal subgroup of $G$. This is clear if $\alpha=0$, so assume that $\alpha>0$ and $Z_\beta(G)$ is normal in $G$ for every ordinal $\beta<\alpha$. If $\alpha$ is a successor ordinal then certainly $Z_\alpha(G)$ is normal in $G$, by the correspondence theorem. Suppose $\alpha$ is a limit ordinal. By assumption $G$ is connected, hence $G^1$ is affine by \cite[Cor.\ 8.11]{Mi}. 
Then applying \cite[Prop.\ 2.2]{S} to $G^1$ implies that $Z_\alpha(G^1)$ is a normal subgroup of $G^1$. Therefore $Z_\alpha(G)$ is a normal subgroup of $G$, as it is the preimage of $Z_\alpha(G^1)$ under the natural projection $\zeta:G\to G^1$.

\vspace{2mm}\noindent The (\textit{transfinitely extended}) \textit{upper central series} of $G$ is the central series \begin{equation}\label{transuppercentser} \{Z_\alpha(G)\}\,,\end{equation} where $\alpha$ runs over all of the ordinals. 

\vspace{2mm}\noindent Let $\kappa$ be the cardinality of our scheme $G$ (in the sense of \cite[000H]{Stacks}). Since the cardinality of each $Z_\alpha(G)$ cannot exceed $\kappa$, there exists a least ordinal $\lambda$ 
such that the upper central series $(\ref{transuppercentser})$ of $G$ terminates at $Z_\lambda(G)$. That is, $Z_\lambda(G)=Z_\mu(G)$ for all ordinals $\mu >\lambda$. We will show in Theorem \ref{UCSterminatesthm}(a) that $\lambda$ is strictly less than the ordinal $\omega^2$; in particular, it is countable.

\vspace{2mm}\noindent We define the \textit{hypercenter} of $G$ to be the normal subgroup \begin{equation}\label{hypercenterdef} Z_\infty(G):=Z_\lambda(G)\, .\end{equation} A subgroup of $G$ is said to be \textit{hypercentral} if it is contained in $Z_\infty(G)$. 

\vspace{2mm}\noindent We conclude this section with a few basic results that we will need to prove Theorem \ref{centerlessbynilpotent}. Their proofs are straightforward. Continue to let $G$ be a connected algebraic $k$-group.

\begin{lemma}\label{simpleobservation} For each ordinal $\alpha$ and each integer $i\geq 0$, we have $$Z_{\alpha+i}(G)/Z_\alpha(G)=Z_i(G/Z_\alpha(G)).$$
\begin{proof} For each ordinal $\beta$, let $\zeta_\beta:G\to G/Z_\beta(G)=:G^\beta$ be the natural projection. For each ordinal $\beta$ and integer $j\geq 0$, let $\smash{\rho^\beta_j}:G^\beta\to G^\beta/Z_j(G^\beta)$ be the natural projection.

\vspace{2mm}\noindent Now fix an ordinal $\alpha$ and an integer $i\geq 0$. We will show that $\zeta_{\alpha+i}=\rho^\alpha_i\circ\zeta_\alpha$. We proceed by induction on $i$. The cases $i=0$ and $i=1$ are clear, so assume $i>1$. By the inductive hypothesis, we have $$\zeta_{\alpha+i-1}=\rho^\alpha_{i-1}\circ\zeta_\alpha:G\to G^{\alpha+i-1}.$$ But by definition (or again by the inductive hypothesis), we see that $\zeta_{\alpha+i}$ equals the composition of $\zeta_{\alpha+i-1}:G\to G^{\alpha+i-1}$ with the natural projection $\rho^{\alpha+i-1}_1:G^{\alpha+i-1}\to G^{\alpha+i-1}/Z(G^{\alpha+i-1})$. Similarly, $\rho^{\alpha+i-1}_1\circ\rho^\alpha_{i-1}=\rho^\alpha_i$. In summary, $$\zeta_{\alpha+i}=\rho^{\alpha+i-1}_1\circ\zeta_{\alpha+i-1}=\rho^{\alpha+i-1}_1\circ\rho^\alpha_{i-1}\circ \zeta_\alpha=\rho^\alpha_i\circ\zeta_\alpha.$$ Consequently, $$Z_{\alpha+i}(G)/Z_\alpha(G)=\zeta_\alpha(\ker\zeta_{\alpha+i})=\ker\rho^\alpha_i=Z_i(G^\alpha). \qedhere$$ 
\end{proof}
\end{lemma} 

\begin{lemma}\label{hypercenterchar} The hypercenter $Z_\infty(G)$ equals the intersection of all normal subgroups $N$ of $G$ such that $G/N$ has trivial center. 
\begin{proof} By construction, the center of $G/Z_\infty(G)$ is trivial. Let $N$ be a normal subgroup of $G$ such that $Z(G/N)$ is trivial. We need to show that $Z_\infty(G)$ is contained in $N$. Assume otherwise (for a contradiction). Then there exists a least ordinal $\alpha$ such that $Z_\alpha(G)$ is not contained in $N$. 

\vspace{2mm}\noindent Suppose $\alpha$ is a limit ordinal. Then by definition $Z_\alpha(G)=\bigcup_{\beta<\alpha}Z_\beta(G)$, but each $Z_\beta(G)$ is contained in $N$ hence so is $Z_\alpha(G)$. This is a contradiction.

\vspace{2mm}\noindent Now suppose $\alpha$ is a successor ordinal. Let $\varphi:G\to G/N$ be the natural projection. Since $Z_{\alpha-1}(G)$ is contained in $N$, by the universal property of the quotient $\varphi$ factors through a quotient map $\psi:G/Z_{\alpha-1}(G)\to G/N$. Observe that $Z(G/Z_{\alpha-1}(G))\subseteq\ker\psi$, since $\psi(Z(G/Z_{\alpha-1}(G)))\subseteq Z(G/N)=1$. But by definition $Z_\alpha(G)/Z_{\alpha-1}(G)=Z(G/Z_{\alpha-1}(G))$, and hence $Z_{\alpha}(G)\subseteq N$. Once again, we have a contradiction.
\end{proof}
\end{lemma}

\begin{corollary}\label{hypercentercharcor} Let $f:G\to G'$ be a faithfully flat homomorphism between connected algebraic $k$-groups, where $\ker f$ is hypercentral in $G$. Then $f(Z_\infty(G))=Z_\infty(G')$. 
\begin{proof} Consider the natural isomorphism $G/f^{-1}(Z_\infty(G'))\cong G'/Z_\infty(G')$. By construction $G'/Z_\infty(G')$ has trivial center, so this is also true of $G/f^{-1}(Z_\infty(G'))$. Therefore $Z_\infty(G)\subseteq f^{-1}(Z_\infty(G'))$, by Lemma \ref{hypercenterchar}.

\vspace{2mm}\noindent Now for the opposite inclusion. Since $\ker f$ is hypercentral in $G$, we have an isomorphism $G/Z_\infty(G)\cong G'/f(Z_\infty(G))$. Therefore $G'/f(Z_\infty(G))$ has trivial center, as this is true of $G/Z_\infty(G)$. Hence $Z_\infty(G')\subseteq f(Z_\infty(G))$, again by Lemma \ref{hypercenterchar}.
\end{proof}
\end{corollary}


\section{Proofs}\label{proofs}

\noindent Our goal in this section is to prove the main result of this paper; namely, Theorem \ref{centerlessbynilpotent}. For the reader's convenience, we restate it here.

\vspace{2mm}\noindent \textbf{Theorem 1.1.} Let $k$ be a field and let $G$ be a connected algebraic $k$-group. There exists a nilpotent normal subgroup $Z_\infty(G)$ of $G$ such that the center of $G/Z_\infty(G)$ is trivial. The formation of $Z_\infty(G)$ commutes with base change by algebraic field extensions.

\vspace{2mm}\noindent We prove Theorem \ref{centerlessbynilpotent} via a series of intermediate results, which include Theorems \ref{largestnilpotentnormsubgp}, \ref{hypercenterthm}, and \ref{UCSterminatesthm}, along with several lemmas. Along the way, we prove Corollaries \ref{hypercentercor1} and \ref{hypercentercor2}. We conclude by giving an example which shows that Theorem \ref{centerlessbynilpotent} fails without the connectedness assumption on $G$.

\vspace{2mm}\noindent Our setup in this section is as follows.

\vspace{2mm}\noindent Let $k$ be a field, say of characteristic $p\geq 0$. Let $G$ be a connected algebraic $k$-group (not necessarily affine, unless explicitly stated). We continue to use the notation that was introduced in \S \ref{upcenser}. In particular, let \begin{equation}\label{centralquotientmap} \zeta:G\to G/Z(G)=:G^1\end{equation} denote the natural projection.

\vspace{2mm}\noindent It is known that $G^1$ is affine and its center $Z(G^1)$ is unipotent (this result is due to Rosenlicht, refer to \cite[Cor.\ 8.11]{Mi} 
and \cite[XVII, Lem.\ 7.3.2]{SGA3}). 
We will repeatedly use this result without any further reference.

\vspace{2mm}\noindent \underline{Proof of Theorem \ref{largestnilpotentnormsubgp}.}

\vspace{2mm}\noindent We aim to show that $G$ admits a largest nilpotent normal subgroup.

\begin{proof} Recall that $G^1$ is affine and $Z(G^1)$ is unipotent. According to \cite[Thm.\ 1.1]{S}, there exists a largest unipotent normal subgroup $\Rad_u(G^1)$ of $G^1$. Let $F(G)$ be the preimage of $\Rad_u(G^1)$ under the central quotient $\zeta:G\to G^1$, i.e. $$F(G):=\zeta^{-1}(\Rad_u(G^1)).$$ Certainly $F(G)$ is a nilpotent normal subgroup of $G$. We claim that it is the largest such subgroup.

\vspace{2mm}\noindent Let $H$ be a nilpotent normal subgroup of $G$. Then $\zeta(H)$ is a nilpotent normal subgroup of $G^1$. Since $G^1$ is connected and affine and $Z(G^1)$ is unipotent, applying \cite[IV, \S4, Cor.\ 1.13]{DG} tells us that $\zeta(H)$ is unipotent. So $\zeta(H)$ is contained in $\Rad_u(G^1)$. Hence $H\subseteq F(G)$. This completes the proof. 
\end{proof}

\noindent We now move on to the proof of Theorem \ref{hypercenterthm}, and its corollaries. We will need the following two lemmas. 

\begin{lemma}\label{trignilpotent} Let $G\subseteq\GL_d$ be an affine algebraic $k$-group. Let $H$ be a nilpotent subgroup of $G$ which becomes trigonalisable over an algebraic closure of $k$. Then the nilpotency class of $H$ is at most $d(d-1)/2+1$.
\begin{proof} Base changing by an algebraic closure $\overline{k}$ of $k$ leaves the nilpotency class of $H$ unchanged, so without loss of generality we can and do assume $k=\overline{k}$. Then by \cite[Thm.\ 16.26(a)]{Mi} we may write $H=U\rtimes T$, where $U$ is unipotent and $T$ is diagonalisable.

\vspace{2mm}\noindent Since $H$ is trigonalisable, applying \cite[6.49, Thm.\ 16.21 (and its proof)]{Mi} 
tells us that there exists a central series \begin{equation}\label{centunip} U=U_0 \supset U_1\supset...\supset U_r=0\end{equation} which satisfies the following conditions. The length $r$ of this central series $(\ref{centunip})$ is at most $d(d-1)/2$. For each $i\geq 0$ the subgroup $U_i$ is normal in $H$. Finally, for each $i\geq 0$ there exists an embedding of $U_i/U_{i+1}$ into $\mathbb{G}_a$ such that the action of $T$ on $U_i/U_{i+1}$ by conjugation extends to a linear action on $\mathbb{G}_a$. 

\vspace{2mm}\noindent For each $i\geq 0$ we form the semidirect product $(U_i/U_{i+1})\rtimes T$. Observe that $(U_i/U_{i+1})\rtimes T$ is a quotient of the subgroup $U_i\rtimes T$ of $H$. Hence $(U_i/U_{i+1})\rtimes T$ is nilpotent, since $H$ is nilpotent. But this forces each $U_i/U_{i+1}$ to have weight 0 as a $T$-module, as otherwise the center of $(U_i/U_{i+1})\rtimes T$ would be trivial. 
In other words, $T$ acts trivially on each $U_i/U_{i+1}$. So the central series $(\ref{centunip})$ extends to a central series \begin{equation}\label{centtrig} H\supset U\supset U_1\supset...\supset U_r=0\,. \end{equation} 
We deduce that the nilpotency class of $H$ is at most $r+1$.
\end{proof}
\end{lemma}

\begin{lemma}\label{largestmulttypenormsubgp} Let $G$ be a connected affine algebraic $k$-group. All multiplicative type normal subgroups of $G$ are contained in $Z(G)_s$. 
Moreover, $Z(G/Z(G)_s)_s$ is trivial.
\begin{proof} Let $M$ be a multiplicative type normal subgroup of $G$. Then $M$ is central in $G$, by \cite[Cor.\ 12.38]{Mi}. Hence $M\subseteq Z(G)_s$. 
Since $Z(G/Z(G)_s)_s$ is a nilpotent normal subgroup of $G/Z(G)_s$, it is unipotent (and hence trivial) by \cite[IV, \S4, Cor.\ 1.13]{DG}. 
\end{proof}
\end{lemma}

\noindent \underline{Proof of Theorem \ref{hypercenterthm}.}

\vspace{2mm}\noindent Recall that $G$ is a connected algebraic $k$-group.

\begin{proof} (a). We first consider the case where $G$ is affine. For each integer $i\geq 0$, the $i$'th center $Z_i(G)$ is a normal subgroup of $G$. Then $Z_\omega(G)$ is a normal subgroup of $G$, by \cite[Prop.\ 2.2]{S}. We know from Theorem \ref{largestnilpotentnormsubgp} that there exists a largest nilpotent normal subgroup $F(G)$ of $G$. Each $Z_i(G)$ is nilpotent (with nilpotency class at most $i$), 
hence it is contained in $F(G)$. Therefore $Z_\omega(G)\subseteq F(G)$, by definition of the schematic union. So indeed $Z_\omega(G)$ is nilpotent.

\vspace{2mm}\noindent We now relax the assumption that $G$ is affine. Recall that $\zeta:G\to G/Z(G)=:G^1$ denotes the natural projection, and that $G^1$ is affine. For every integer $i\geq 1$ observe that $$Z_i(G)=\zeta^{-1}(Z_{i-1}(G^1)),$$ by definition of the upper central series. It follows that \begin{equation}\label{pulling} Z_\omega(G)=\mathsmaller{\bigcup}_{i\geq 1}\,\zeta^{-1}(Z_{i-1}(G^1))=\zeta^{-1}(Z_{\omega}(G^1)),\end{equation} since $\zeta$ is a quotient map. We already proved the affine case, so $Z_{\omega}(G^1)$ is a nilpotent normal subgroup of $G^1$. Therefore $Z_\omega(G)$ is a nilpotent normal subgroup of $G$, since it is a central extension of $Z_{\omega}(G^1)$. This proves (a).

\vspace{2mm}\noindent (b). Let $K/k$ be an algebraic field extension. Recall that short exact sequences of algebraic $k$-groups are preserved by base change by $K/k$. 
Certainly the formation of the center $Z(G)$ of $G$ commutes with base change by $K/k$, i.e. $Z(G)_K=Z(G_K)$. Consequently, $Z_i(G)_K=Z_i(G_K)$ for every integer $i\geq 0$.

\vspace{2mm}\noindent If $G$ is affine then we may apply Proposition \ref{schunionbasechange}, giving us \begin{equation}\label{affinepresbase} Z_\omega(G)_K=\mathsmaller{\bigcup}_{i\geq 0}\,(Z_i(G)_K)=\mathsmaller{\bigcup}_{i\geq 0}\,Z_i(G_K)=Z_{\omega}(G_K).\end{equation}

\vspace{1mm}\noindent Now suppose $G$ is arbitrary. We know $G^1$ is affine, so $Z_{\omega}(G^1)_K=Z_{\omega}((G^1)_K)$ by Equation $(\ref{affinepresbase})$. Then pulling back by $\zeta$ and using Equation $(\ref{pulling})$, we deduce that $Z_\omega(G)_K=Z_{\omega}(G_K)$. This proves (b).

\vspace{2mm}\noindent (c). Suppose $G$ is affine. Recall that $Z(G)_s$ denotes the largest multiplicative type subgroup of $Z(G)$. We will show that \begin{equation}\label{redtounip} Z(G)_s=Z(Z_\omega(G))_s.\end{equation}

\vspace{1mm}\noindent By definition $Z(G)\subseteq Z_\omega(G)$, so certainly $Z(G)_s\subseteq Z(Z_\omega(G))_s$. Now for the opposite inclusion. By \cite[IV, \S 3, Thm.\ 1.1(a)]{DG}, $Z(Z_\omega(G))_s$ is a characteristic subgroup of $Z(Z_\omega(G))$. We know from part (a) that $Z_\omega(G)$ is normal in $G$, hence $Z(Z_\omega(G))_s$ is normal in $G$. Therefore $Z(Z_\omega(G))_s$ is central in $G$, by \cite[Cor.\ 12.38]{Mi}. So indeed $Z(G)_s$ contains $Z(Z_\omega(G))_s$, establishing Equation $(\ref{redtounip})$. 

\vspace{2mm}\noindent Now suppose in addition that $Z(G)$ is unipotent. Then $$Z(Z_\omega(G))_s=Z(G)_s=1,$$ by Equation $(\ref{redtounip})$. Again by part (a), we know that $Z_\omega(G)$ is a nilpotent normal subgroup of $G$. Since $G$ is affine and $Z(Z_\omega(G))$ is unipotent, by \cite[IV, \S4, Thm.\ 1.10]{DG} we deduce that $Z_\omega(G)$ is unipotent. 

\vspace{2mm}\noindent (d). By part (a), $Z_\omega(G)$ is a normal subgroup of $G$. So let \begin{equation}\label{hypercentralquotient} \pi:G\to G/Z_\omega(G)=:G^\omega \end{equation} be the natural projection. By part (b), without loss of generality we can assume that $k$ is algebraically closed. 

\vspace{2mm}\noindent We first consider the case where $G$ is affine and $Z(G)$ is unipotent. Consider the short exact sequence \begin{equation}\label{sesone} 1\to Z_\omega(G)\to\pi^{-1}(Z(G^\omega)_s)\to Z(G^\omega)_s\to 1.\end{equation}

\noindent We know that $Z_\omega(G)$ is unipotent, by part (c). Since $k$ is algebraically closed, the sequence $(\ref{sesone})$ splits by \cite[XVII, Thm.\ 5.1.1(i)(a)]{SGA3}. That is, $\pi^{-1}(Z(G^\omega)_s)=Z_\omega(G)\rtimes T$ for some multiplicative type subgroup $T$ of $G$ which satisfies $T\cong Z(G^\omega)_s$. Note that $Z_\omega(G)\rtimes T$ is a normal subgroup of $G$, since $Z(G^\omega)_s$ is a normal subgroup of $G^\omega$. 

\vspace{2mm}\noindent We claim that $T$ is trivial. To see this, consider the chain \begin{equation}\label{sestwo} Z_1(G)\rtimes T\subseteq Z_2(G)\rtimes T\subseteq ... \end{equation} of 
subgroups of $G$. The schematic union of this chain $(\ref{sestwo})$ is \begin{equation}\label{schunses} Z_\omega(G)\rtimes T=\mathsmaller{\bigcup}_{i\geq 1}(Z_i(G)\rtimes T).\end{equation}

\vspace{1mm}\noindent Let $i\geq 1$ be an integer. Observe that the $i$'th center of $Z_i(G)\rtimes T$ contains $Z_i(G)$. 
So the quotient $(Z_i(G)\rtimes T)/Z_i(Z_i(G)\rtimes T)$ is commutative. It follows that $Z_i(G)\rtimes T$ is nilpotent (with nilpotency class at most $i+1$), by definition of the upper central series. 
Note that $Z_i(G)\rtimes T$ is trigonalisable, since $Z_i(G)$ is unipotent. Then applying Lemma \ref{trignilpotent} tells us that the nilpotency class of $Z_i(G)\rtimes T$ is bounded above by an integer $c:=c(G)$ which is independent of $i$ (we may take $c=d(d-1)/2+1$, where $d$ is the minimal dimension of a faithful linear representation of $G$).

\vspace{2mm}\noindent In summary, we have thus far shown that each term of the chain $(\ref{sestwo})$ is nilpotent with nilpotency class at most $c$. Then applying Proposition \ref{nilclasspresschunion} says that its schematic union $Z_\omega(G)\rtimes T$ is also nilpotent with nilpotency class at most $c$. We already observed that $Z_\omega(G)\rtimes T$ is a normal subgroup of $G$, hence $Z_\omega(G)\rtimes T$ is unipotent by \cite[IV, \S4, Cor.\ 1.13]{DG}. Thus $T$ is trivial, proving the claim. In other words, we have shown that $Z(G^\omega)$ is unipotent. Then by part (c), we deduce that $Z_{\omega}(G^\omega)$ is unipotent.

\vspace{2mm}\noindent It remains to consider the general case; that is, $G$ is an arbitrary connected algebraic $k$-group. Recall that $G^1$ is affine and $Z(G^1)$ is unipotent. We already proved this special case, hence $Z_{\omega}(G^1/Z_{\omega}(G^1))$ is unipotent. Recall from Equation $(\ref{pulling})$ that $Z_\omega(G)=\zeta^{-1}(Z_{\omega}(G^1))$. Consequently, $$G^1/Z_{\omega}(G^1)\cong G/Z_\omega(G)=:G^\omega\,.$$ 
Hence $Z_{\omega}(G^\omega)$ is unipotent. This completes the proof of the theorem.
\end{proof}

\noindent \underline{Proof of Corollary \ref{hypercentercor1}.}

\begin{proof} Continue to let $G$ be a connected algebraic $k$-group. Fix an integer $i\geq 1$. Recall that $G^i$ is defined inductively by $G^i:=G^{i-1}/Z(G^{i-1})$. 
So $G^i$ is affine and $Z(G^i)$ is unipotent. 
The result then follows immediately from Theorem \ref{hypercenterthm}(c).
\end{proof}

\noindent \underline{Proof of Corollary \ref{hypercentercor2}.}

\begin{proof} Let $G$ be a connected affine algebraic $k$-group. We know that $Z(G/Z(G)_s)$ is unipotent, by Lemma \ref{largestmulttypenormsubgp}. The result then follows from Theorem \ref{hypercenterthm}(c).
\end{proof}

\noindent \underline{Proof of Theorem \ref{UCSterminatesthm}.}

\vspace{2mm}\noindent Let $G$ be a connected algebraic $k$-group.

\begin{proof} (a). Let $\lambda$ be the least ordinal such that the (transfinitely extended) upper central series $(\ref{transuppercentserorig})$ of $G$ terminates at $Z_\lambda(G)$ (the existence of $\lambda$ is clear, refer to \S\ref{upcenser}). That is, $Z_\lambda(G)=Z_\mu(G)$ for all ordinals $\mu >\lambda$. Our goal is to show that $\lambda < \omega^2$. 

\vspace{2mm}\noindent Fix an integer $n\geq 0$, and consider the ordinals $\omega\cdot n$ and $\omega\cdot (n+1)$. Since $\omega\cdot (n+1)$ is a limit ordinal, by definition we have \begin{equation}\label{onefunfun} Z_{\omega\cdot(n+1)}(G)=\mathsmaller{\bigcup}_{\beta < \omega\cdot(n+1)}Z_\beta(G)=\mathsmaller{\bigcup}_{i\geq 0}Z_{\omega\cdot n+i}(G),\end{equation} where $i$ runs over the non-negative integers. Denote $G^{\omega\cdot n}:=G/Z_{\omega\cdot n}(G)$. Then modding out Equation $(\ref{onefunfun})$ by $Z_{\omega\cdot n}(G)$ gives us \begin{equation}\label{twofunfun} Z_{\omega\cdot(n+1)}(G)/Z_{\omega\cdot n}(G)=\mathsmaller{\bigcup}_{i\geq 0}(Z_{\omega\cdot n+i}(G)/Z_{\omega\cdot n}(G))=\mathsmaller{\bigcup}_{i\geq 0}Z_i(G^{\omega\cdot n})=:Z_\omega(G^{\omega\cdot n}), \end{equation} using Lemma \ref{simpleobservation} and the fact that the formation of schematic unions commutes with taking quotient maps.

\vspace{2mm}\noindent Assume (for a contradiction) that, for every integer $j\geq 0$, $Z_\omega(G^{\omega\cdot j})$ is not finite. Since our group scheme $G$ (and thus each $Z_\omega(G^{\omega\cdot j})$) is of finite type over $k$, this implies that each $Z_\omega(G^{\omega\cdot j})$ is positive-dimensional. Hence by Equation $(\ref{twofunfun})$ we have $$\dim Z_{\omega\cdot j}(G)\lneq \dim Z_{\omega\cdot(j+1)}(G),$$ for each integer $j\geq 0$. But $G$ is finite-dimensional and $\dim Z_{\omega\cdot j}(G)\leq \dim G$ for each $j\geq 0$, which is a contradiction.

\vspace{2mm}\noindent So there exists some integer $m\geq 0$ such that $Z_\omega(G^{\omega\cdot m})$ is finite. This implies that the (non-extended) upper central series \begin{equation}\label{rawh} 1=Z_0(G^{\omega\cdot m})\subseteq Z_1(G^{\omega\cdot m})\subseteq ...\end{equation} of $G^{\omega\cdot m}$ terminates after finitely many terms. That is, there exists an integer $t\geq 0$ such that $Z_t(G^{\omega\cdot m})=Z_{t+1}(G^{\omega\cdot m})$. Then again using Lemma \ref{simpleobservation}, we deduce that \begin{equation}\label{finalconc} Z_{\omega\cdot m+t}(G)=Z_{\omega\cdot m+t+1}(G).\end{equation} Therefore $\lambda\leq \omega\cdot m+t<\omega^2$. In particular, $\lambda$ is countable. This proves (a).

\vspace{2mm}\noindent (b). Let $K/k$ be an algebraic field extension. Let $\alpha$ be an ordinal. We show that $Z_\alpha(G_K)=Z_\alpha(G)_K$ by transfinite induction on $\alpha$. If $\alpha=0$ then this is clear, so assume $\alpha>0$.

\vspace{2mm}\noindent Suppose $\alpha$ is a successor ordinal. Then by definition $Z_{\alpha}(G)/Z_{\alpha-1}(G)=Z(G/Z_{\alpha-1}(G))$. Recall that short exact sequences of algebraic $k$-groups are preserved by base change by $K/k$, as is the formation of the center. We know that $Z_{\alpha-1}(G_K)=Z_{\alpha-1}(G)_K$, by the inductive hypothesis. Therefore $Z_\alpha(G_K)=Z_\alpha(G)_K$.

\vspace{2mm}\noindent Now suppose $\alpha$ is a limit ordinal. By part (a) it suffices to assume $\alpha=\omega\cdot n$ for some integer $n\geq 1$. The case $n=1$ was proved in Theorem \ref{hypercenterthm}(b), so assume $n>1$. The result then follows from combining Equation $(\ref{twofunfun})$ with the inductive hypothesis.
\end{proof}

\noindent We now have enough to prove Theorem \ref{centerlessbynilpotent}. We give a proof which avoids using transfinite induction, in order to keep things as elementary as possible. 

\vspace{2mm}\noindent \underline{Proof of Theorem \ref{centerlessbynilpotent}.}

\vspace{2mm}\noindent Let $G$ be a connected algebraic $k$-group.

\begin{proof} Consider the hypercenter $Z_\infty(G)$ of $G$, as defined in \S \ref{upcenser}. By construction $Z_\infty(G)$ is a normal subgroup of $G$, and the quotient $G/Z_\infty(G)$ has trivial center. Our goal is to show that $Z_\infty(G)$ is nilpotent. Continue to let $$\pi:G\to G/Z_\omega(G)=:G^\omega$$ denote the natural projection, as in $(\ref{hypercentralquotient})$.

\vspace{2mm}\noindent We first consider the case where $G$ is affine and $Z(G)$ is unipotent. For this case we claim that $Z_\infty(G)$ is unipotent. We prove this claim by induction on the dimension $\dim G$ of $G$.

\vspace{2mm}\noindent First consider the case where the upper central series of $G$ terminates after finitely many terms. That is, there exists an integer $j\geq 0$ such that $Z_\infty(G)=Z_\omega(G)=Z_j(G)$. We know from Theorem \ref{hypercenterthm}(c) that $Z_\omega(G)$ is unipotent, and so we are done.

\vspace{2mm}\noindent Henceforth assume that the upper central series of $G$ does not terminate after finitely many terms. Then by definition $Z_\omega(G)$ is not finite. Since $Z_\omega(G)$ is a Noetherian scheme, 
this means that $\dim Z_\omega(G)>0$. 
Consequently, $\dim G^\omega\lneq \dim G$. 

\vspace{2mm}\noindent We now proceed with the induction. We already took care of the case $\dim G=0$, so assume otherwise. Observe that $Z(G^\omega)$ is unipotent, by Theorem \ref{hypercenterthm}(d). Since $\dim G^\omega\lneq \dim G$ we may apply the inductive hypothesis to $G^\omega$, which tells us that $Z_\infty(G^\omega)$ is unipotent.

\vspace{2mm}\noindent Since $\ker\pi=Z_\omega(G)$ is hypercentral in $G$, by Corollary \ref{hypercentercharcor} we have $$Z_\infty(G)=\pi^{-1}(Z_\infty(G^\omega)).$$ 
According to Theorem \ref{hypercenterthm}(c), $Z_\omega(G)$ is unipotent. Since the property of unipotency is preserved by taking group extensions, we deduce that $Z_\infty(G)$ is unipotent. This proves the claim. 

\vspace{2mm}\noindent We now move on to the general case; that is, $G$ is an arbitrary connected algebraic $k$-group (not necessarily affine). Consider the central quotient map $\zeta:G\to G/Z(G)=:G^1$. Recall that $G^1$ is affine and its center $Z(G^1)$ is unipotent (as previously discussed, this result is due to Rosenlicht). Then, by the aforementioned claim, $Z_\infty(G^1)$ is unipotent.

\vspace{2mm}\noindent Again by Corollary \ref{hypercentercharcor}, we have $$Z_\infty(G)=\zeta^{-1}(Z_\infty(G^1)).$$ Hence $Z_\infty(G)$ is nilpotent, as it is a central extension of $Z_\infty(G^1)$.

\vspace{2mm}\noindent The second assertion follows immediately from Theorem \ref{UCSterminatesthm}.
\end{proof}

\noindent We conclude by observing that Theorem \ref{centerlessbynilpotent} (along with some of our other results) fails without the connectedness assumption on $G$. We illustrate this with the following example. 

\begin{example}\label{etaleunion} Let $k$ be an algebraically closed field of characteristic $p\neq 2$. Let $$G:=\mathbb{G}_m\rtimes\mathbb{Z}/2\mathbb{Z}\, ,$$ where $\mathbb{Z}/2\mathbb{Z}$ acts on $\mathbb{G}_m$ by inversion. For each integer $i\geq 0$, the $i$'th center $Z_i(G)$ of $G$ equals $\mu_{2^i}$ (i.e. the $(2^i)$'th roots of unity inside $\mathbb{G}_m$). 
So the upper central series of $G$ does not stabilise after finitely many terms. 
Observe that $$Z_\omega(G)=\mathsmaller{\bigcup}_{i\geq 0}\,\,\mu_{2^i}=\mathbb{G}_m,$$ yet $Z_{\omega}(G/Z_\omega(G))\cong\mathbb{Z}/2\mathbb{Z}$ is not unipotent since $p\neq 2$. 

\vspace{2mm}\noindent Now let $N$ be a nilpotent normal subgroup of $G$. It is easy to see that all proper normal subgroups of $G$ are contained in $\mathbb{G}_m$. 
Certainly $G$ is not nilpotent, so $N$ is contained in $\mathbb{G}_m$. If $N=\mathbb{G}_m$ then $G/N\cong\mathbb{Z}/2\mathbb{Z}$. Otherwise, $N$ is finite and $G/N\cong G$. In either case, $Z(G/N)$ is non-trivial.
\end{example} 

\bigskip\noindent
{\textbf {Acknowledgements}}:
The author is employed/funded by the University of Freiburg Mathematical Institute.

\bibliographystyle{amsalpha}

\newcommand{\etalchar}[1]{$^{#1}$}
\providecommand{\bysame}{\leavevmode\hbox to3em{\hrulefill}\thinspace}
\providecommand{\MR}{\relax\ifhmode\unskip\space\fi MR }
\providecommand{\MRhref}[2]{%
	\href{http://www.ams.org/mathscinet-getitem?mr=#1}{#2} }
\providecommand{\href}[2]{#2}

\end{document}